\documentclass[11pt]{amsart}
\usepackage{amsfonts}
\usepackage{amsmath}
\usepackage{amsthm}
\usepackage{url}
\usepackage[all]{xy}
\usepackage{latexsym}
\usepackage{amssymb}
\usepackage[cp850]{inputenc}
\usepackage{amsfonts}
\usepackage{graphicx}
\usepackage{dsfont}
\usepackage{float}

\usepackage[usenames]{color}

\newtheorem{sat}{Theorem}[section]		
\newtheorem{lem}[sat]{Lemma}
\newtheorem{kor}[sat]{Corollary}			
\newtheorem{prop}[sat]{Proposition}

\newtheorem*{defi*}{Definition}			
\newtheorem*{bei*}{Example}
\newtheorem*{sat*}{Theorem}				
\newtheorem*{kor*}{Corollary}
\newtheorem*{rmk*}{Remark}				
	
\newtheorem*{quest*}{Question}	
	
\newtheorem{claim}{Claim}	

\let\ssection=\section
\renewcommand{\section}{\setcounter{equation}{0}\ssection}

\newtheorem*{namedtheorem}{\theoremname}
\newcommand{\theoremname}{testing}
\newenvironment{named}[1]{\renewcommand{\theoremname}{#1}\begin{namedtheorem}}{\end{namedtheorem}}

\theoremstyle{remark}
\newtheorem*{bem}{Remark}
\newtheorem{bei}{Example}

\newtheorem*{namedtheoremr}{\theoremnamer}
\newcommand{\theoremnamer}{testing}

			\newcommand{\BH}{\mathbb H}
\newcommand{\BR}{\mathbb R}			
			
			\newcommand{\BZ}{\mathbb Z}

\newcommand{\CA}{\mathcal A}		
		\newcommand{\calD}{\mathcal D}

\newcommand{\CI}{\mathcal I}		\newcommand{\CJ}{\mathcal J}
		
		\newcommand{\CN}{\mathcal N}
		
		\newcommand{\CR}{\mathcal R}

\newcommand{\actson}{\curvearrowright}

\newcommand{\bs}{\backslash}

\DeclareMathOperator{\SL}{SL}		
\DeclareMathOperator{\PSL}{PSL}		
\DeclareMathOperator{\Id}{Id}		
\DeclareMathOperator{\vol}{vol}		
\DeclareMathOperator{\tr}{Tr}

\newcommand{\comment}[1]{}

\DeclareMathOperator{\Stab}{Stab}

\DeclareMathOperator{\supp}{Supp}

\newcommand{\fsubd}{\mathrel{{\scriptstyle\searrow}\kern-1ex^d\kern0.5ex}}
\newcommand{\bsubd}{\mathrel{{\scriptstyle\swarrow}\kern-1.6ex^d\kern0.8ex}}

\renewcommand{\epsilon}{\varepsilon}
\renewcommand{\le}{\leqslant}
\renewcommand{\ge}{\geqslant}
\renewcommand{\emptyset}{\varnothing}

\begin{document}

\title[]{Counting and equidistribution of reciprocal geodesics and dihedral groups}
  \author{Viveka Erlandsson}
  \address{School of Mathematics, University of Bristol \\ Bristol BS8 1UG, UK {\rm and}  \newline ${ }$ \hspace{0.2cm} Department of Mathematics and Statistics, UiT The Arctic University of  \newline ${ }$ \hspace{0.2cm} Norway}
  \email{v.erlandsson@bristol.ac.uk}
\thanks{The first author gratefully acknowledges support from EPSRC grant EP/T015926/1 and UiT Aurora Center for Mathematical Structures in Computations (MASCOT)}
\author{Juan Souto}
\address{UNIV RENNES, CNRS, IRMAR - UMR 6625, F-35000 RENNES, FRANCE}
\email{jsoutoc@gmail.com}

\begin{abstract}
We study the growth of the number of conjugacy classes of infinite dihedral subgroups of lattices in $\PSL_2\BR$, generalizing earlier work of Sarnak \cite{Sarnak} and Bourgain-Kontorovich \cite{Bourgain-Kontorovich} on the growth of the number of reciprocal geodesics on the modular surface. We also prove that reciprocal geodesics are equidistributed in the unit tangent bundle.
\end{abstract}

\maketitle

\section{}

In this note we are interested in counting conjugacy classes of {\em infinite dihedral subgroups}, that is subgroups isomorphic to $(\BZ/2\BZ)*(\BZ/2\BZ)$, of lattices $\Gamma\subset\PSL_2\BR$. Evidently, we will only care about lattices $\Gamma$ with 2-torsion, that is lattices that have elements of order two---otherwise $\Gamma$ has no infinite dihedral subgroups. With the action of $\Gamma$ on the hyperbolic plane $\BH^2$ in mind, we refer to the elements of order two as {\em involutions}. 

Discrete infinite dihedral subgroup of $\PSL_2\BR$, for example those which arise as subgroups of a lattice, preserve a unique geodesic $\CA_D$ in $\BH^2$, the {\em axis} of $D$. We will refer to the length $\ell(D)$ of the quotient $D\bs\CA_D$ as the {\em length} of $D$. Our first goal is to study the asymptotic behaviour of the number of conjugacy classes of dihedral subgroups of $\Gamma$ of at most length $L$, but before stating a precise result we need some notation that will be used throughout the paper. We will denote by 
$$\CI_\Gamma=\{\gamma\in\Gamma\setminus\Id,\ \gamma^2=\Id\}$$ 
the set of involutions in $\Gamma$, by 
$$\calD_\Gamma=\{\text{subgroups }D\subset\Gamma\text{ isomorphic to }(\BZ/2\BZ)*(\BZ/2\BZ)\}$$
the set of all infinite dihedral subgroups of $\Gamma$, and by
$$\calD_\Gamma(L)=\{D\in\calD_\Gamma\text{ with }\ell(D)\le L\}$$
that consisting of infinite dihedral subgroups of length at most $L$. Note that the action of $\Gamma$ on itself by conjugation induces actions on $\CI_\Gamma$, $\calD_\Gamma$, and $\calD_\Gamma(L)$---normalizers $\CN_\Gamma(\cdot)$ become then stablizers. 

\begin{sat}\label{sat sarnak}
For every lattice $\Gamma\subset\PSL_2\BR$ with 2-torsion we have
$$\vert\Gamma\bs\calD_\Gamma(L)\vert\sim \frac{C(\Gamma)}{\vert\chi^{or}(\Gamma\bs\BH^2)\vert}\cdot e^L$$
where $\chi^{or}(\Gamma\bs\BH^2)$ is the Euler characteristic of the orbifold $\Gamma\bs\BH^2$, where 
$$C(\Gamma)=\frac 14\cdot\left(\sum_{\sigma\in\Gamma\bs\CI_\Gamma}\frac 1{\vert\CN_\Gamma(\sigma)\vert}\right)^2$$
and $\sim$ means that the ratio between both quatities tends to $1$ when $L\to\infty$.
\end{sat}

Theorem \ref{sat sarnak} generalizes a result of Sarnak. Recall namely that a hyperbolic element $\gamma$ in $\Gamma\subset\PSL_2\BR$ is {\em reciprocal} if it is conjugated to its inverse, that is if there is $\sigma\in \Gamma$ with $\gamma^{-1}=\sigma^{-1}\gamma\sigma$. An unoriented closed geodesic in the orbifold $\Gamma\backslash\BH^2$ is {\em reciprocal} if its free homotopy class is represented by a reciprocal element in $\Gamma$. Now, as already pointed out by Fricke and Klein \cite{Fricke-Klein} there is a bijection between (maximal) infinite dihedral subgroups and (primitive) reciprocal geodesics: the (unoriented) reciprocal geodesic in $\Gamma\bs\BH^2$ corresponding to the infinite dihedral group $D$ is the quotient $\gamma_D=T_D\bs\CA_D$ where $T_D$ is the index two subgroup of $D$ consisting of hyperbolic elements and where $\CA_D$ is, as above, the axis of $D$. The length of the dihedral group and the trace of the associated reciprocal geodesic are related by
\begin{equation}\label{eq trace length}
\tr(\gamma_D)=2\cdot\cosh(2^{-1}\ell(\gamma_D))=2\cdot\cosh(\ell(D)).
\end{equation}
and since $2\cdot \cosh(\ell)\sim e^\ell$ for large $\ell$ we get from Theorem \ref{sat sarnak} the following:

\begin{kor}\label{kor reciprocal}
For every lattice $\Gamma\subset\PSL_2\BR$ with 2-torsion we have
$$\vert\{\gamma\text{ reciprocal geodesics in }\Gamma\bs\BH^2\text{ with }\tr(\gamma)\le X\}\vert\sim\frac{C(\Gamma)}{\vert\chi^{or}(\Gamma\bs\BH^2)\vert}\cdot X$$
as $X\to\infty$. Here notation is as in Theorem \ref{sat sarnak}.
\end{kor}

In the particular case that $\Gamma=\PSL_2\BZ$ we have $C(\PSL_2\BZ)=\frac 1{16}$ and $\chi^{or}(\PSL_2\BZ)=\frac{-1}6$, meaning that in the modular surface there are asymptotically $\frac 38X$ reciprocal geodesics with trace at most $X$. This asymptotic was first obtained, among other results, by Sarnak in \cite{Sarnak} and it was Sarnak's paper what got us interested in these matters.

Another paper that motivated us was one by Bourgain and Kontorovich \cite{Bourgain-Kontorovich} giving lower bounds for the number of ``low-lying'' reciprocal geodesics in the modular surface. More precisely they proved that for every $\delta>0$ there is a compact subset $K_\delta\subset\PSL_2\BZ\bs\BH^2$ with 
\begin{equation}\label{eq bourgain kontorovich}
\vert\{\gamma\text{ reciprocal geodesics in }K_\delta\text{ with }\tr(\gamma)\le X\}\vert>X^{1-\delta}
\end{equation}
for all $X>0$ large enough. Again we get a generalization of this theorem:

\begin{sat}\label{sat bourgain kontorovich}
Let $\Gamma$ be a lattice with 2-torsion. For every $\delta>0$ there is a compact set $K_\delta\subset\Gamma\bs\BH^2$ such that
$$\vert\Gamma\bs\{D\in\calD_\Gamma(L)\text{ with }D\bs\CA_D\subset K_\delta\}\vert> e^{(1-\delta)\cdot L}$$
for all $L>0$ large enough.
\end{sat}

While Theorem \ref{sat sarnak} and Theorem \ref{sat bourgain kontorovich} are close to plain vanilla generalizations of the Sarnak and Bourgain-Kontorovich theorems, what is somewhat different are the proofs. Or at least the point of view. Indeed, dropping all number theory from the picture and considering it all just as a geometric problem we reduce both theorems to very classical results on counting lattice points in the hyperbolic plane.

This simplified framework also helps to study how reciprocal geodesics are distributed in the unit tangent bundle $T^1\Gamma\bs\BH^2$. Again we think of them in terms of infinite dihedral subgroups. Although it is unoriented, the reciprocal geodesic $\gamma_D$ associate to the infinite dihedral group $D\in\calD_\Gamma$ corresponds to a unique geodesic flow orbit. We denote by $\vec\gamma_D$ the measure on $T^1\Gamma\bs\BH^2$ given by integrating along this geodesic flow orbit, normalized to have total mass equal to the length of the geodesic. In other words $\vec\gamma_D$ has total measure twice the length $\ell(D)$ of the infinite dihedral group $D$ itself. The behavior of the measures 
\begin{equation}\label{eq measures}
\mu_L=\sum_{D\in\calD_\Gamma(L)}\vec\gamma
\end{equation}
was already consider by Sarnak in \cite{Sarnak}, where he proved that there is a constant $c$ such that for every compact set $\Omega\subset T^1\PSL_2\BZ\backslash\BH^2$ one has
$$\lim\inf_{L\to\infty}\frac 1{\Vert\mu_L\Vert}\mu_L(\Omega)\ge c\cdot\vol(\Omega)$$
where $\vol$ is the probability Liouville measure on the unit tangent bundle, that is the probability measure induced by the Haar measure via the identitfication $T^1\Gamma\bs\BH^2=\Gamma\bs\PSL_2\BR$. 

Sarnak also conjectures in \cite{Sarnak} that, after normalization, the measures $\mu_L$ converge to $\vol$ when $L\to\infty$. This is the statement of the following result:

\begin{sat}\label{sat equidistribution modular surface}
If $\Gamma\subset\PSL_2\BR$ is a lattice which has 2-torsion then the measures $\mu_L$ as in \eqref{eq measures} converge projectively to the Liouville probability measure $\vol$. More precisely we have 
$$\lim_{L\to\infty}\frac 1{\Vert\mu_L\Vert}\int f\ d\mu_L=\int f\ d\vol$$
for every compactly supported continuous function on $T^1\Gamma\bs\BH^2$.
\end{sat}

\begin{bem}
It would be reasonable for the reader to just care about maximal dihedral subgroups, or about primitive reciprocal geodesics, or they might want to replace in \eqref{eq measures} the sum over infinite dihedral subgroups by a sum over reciprocal geodesics. The results stated above remain valid in all those settings because, as we will see below, the proportion of non-maximal dihedral subgroups (resp. non-primitive reciprocal geodesics) among all dihedral subgroups (reciprocal geodesics) with at most length $L$ tends exponentially fast to $0$. 
\end{bem}

Let us now breeze over the organization of the paper. In section \ref{sec dihedral} we recall a few facts about dihedral subgroups of Fuchsian groups, analyzing with some care how the set of conjugacy classes of such subgroups are parametrized by conjugacy classes of pairs of involutions. It follows that to count conjugacy classes of dihedral groups it suffices to count involutions, or rather their fixed points.This is used in section \ref{sec sarnak} to deduce Theorem \ref{sat sarnak} from Delsarte's classical orbit points counting result and in section \ref{sec bourgain kontorovich} to get Theorem \ref{sat bourgain kontorovich} from the fact that lattices in $\PSL_2\BR$ have convex cocompact subgroups with large critical exponent. Still working under the same framework, we prove Theorem \ref{sat equidistribution modular surface} in section \ref{sec uniform}, modulo another equidistribution result whose proof we defer to section \ref{sec mixing}, but which experts will probably consider evident.
\medskip

Before moving on we should mention another paper that got us interested in reciprocal geodesics: in \cite{Ara-Rob} Basmajian--Suzzi Valli prove versions of Sarnak's and Bourgain-Kontorovich's results where trace is replaced by word length with respect to a generating set of the fundamental group of the modular surface. Although we have not pursued this direction, it might well be that the methods we use here can also be used to recover the Basmajian--Suzzi Valli theorems.

\subsection*{Acknowledgements} First and foremost we would like to thank Peter Sarnak for a few really useful and nice e-mails that got us interested in this topic. The second author thanks both the first author's EPSRC grant EP/T015926/1 and the University of Bristol for their hospitality while most of this work was completed. 

\section{}\label{sec dihedral}

Let $\Gamma\subset\PSL_2\BR$ be a discrete, non-elementary, finitely generated subgroup which has 2-torsion. With $\calD_\Gamma$ and $\calD_\Gamma(L)$ as above, let $\calD_\Gamma^{\max}\subset\calD_\Gamma$ and $\calD_\Gamma^{\max}(L)\subset\calD_\Gamma(L)$ be the corresponding sets of maximal infinite dihedral subgroups of $\Gamma$. The goal of this section is to prove the following:

\begin{prop}\label{prop map}
Let $\Gamma\subset\PSL_2\BR$ be a discrete finitely generated subgroup. Suppose that the set $\CI_\Gamma$ of order 2 elements in $\Gamma$ is non-empty, denote by $p_\sigma$ the unique fixed points of $\sigma\in\CI_\Gamma$, and let $\CJ\subset\CI_\Gamma$ be a set of representatives of the set $\Gamma\bs\CI_\Gamma$ of all $\Gamma$-conjugacy classes. The map
\begin{equation}\label{eq map cloud}
\begin{split}
\pi_L:&\bigsqcup_{(\sigma,\bar\sigma)\in\CJ\times\CJ}\left(\Gamma\cdot p_{\bar\sigma}\cap B^*(p_\sigma,L)\right)\to\Gamma\bs\calD_\Gamma(L)\\
\pi_L:&(\sigma,\bar\sigma,\gamma\cdot p_{\bar\sigma})\mapsto\Gamma\text{-conjugacy class of }\langle\sigma,\gamma\bar\sigma\gamma^{-1}\rangle
\end{split}
\end{equation}
is surjective for all $L>0$. Moreover, for $D=\pi_L(\sigma,\bar\sigma,\gamma\cdot p_{\bar\sigma})$ we have 
$$\vert\pi_L^{-1}(D)\vert\le \vert\CN_\Gamma(\sigma)\vert+\vert\CN_\Gamma(\bar\sigma)\vert$$
with equality if $D\in\calD_\Gamma^{\max}$. Here $B^*(p,L)=\{q\in\BH^2\text{ with }0<d_{\BH^2}(p,q)\le L\}$ is the punctured ball of radius $L$ and centered at $p$.
\end{prop}

The interest of Proposition \ref{prop map} is that, as we will exploit in the next sections, it reduces counting dihedral subgroups to counting lattice points. Indeed, the following is an immediate corollary:

\begin{kor}\label{prop enumerate}
Let $\Gamma\subset\PSL_2\BR$ be a discrete finitely generated subgroup. Suppose that the set $\CI_\Gamma$ of order 2 elements in $\Gamma$ is non-empty, denote by $p_\sigma$ the unique fixed points of $\sigma\in\CI_\Gamma$, and let $\CJ\subset\CI_\Gamma$ be a set of representatives of $\Gamma\bs\CI_\Gamma$. Then we have
\begin{align*}
\vert\Gamma\bs\calD_\Gamma(L)\vert&\le \sum_{(\sigma,\bar\sigma)\in\CJ\times\CJ}\vert\Gamma\cdot p_{\bar\sigma}\cap B^*(p_\sigma,L)\vert\\
\vert\Gamma\bs\calD_\Gamma(L)\vert&\ge\sum_{(\sigma,\bar\sigma)\in\CJ\times\CJ}\frac {\vert\Gamma\cdot p_{\bar \sigma}\cap B^*(p_\sigma,L)\vert}{\vert\CN_\Gamma(\sigma)\vert+\vert\CN_\Gamma(\bar \sigma)\vert}\\
\vert\Gamma\bs\calD_\Gamma^{\max}(L)\vert&\le\sum_{(\sigma,\bar \sigma)\in\CJ\times\CJ}\frac {\vert\Gamma\cdot p_{\bar \sigma}\cap B^*(p_\sigma,L)\vert}{\vert\CN_\Gamma(\sigma)\vert+\vert\CN_\Gamma(\bar \sigma)\vert}
\end{align*}
for all $L$. \qed
\end{kor}

The remaining of this section is devoted to prove Proposition \ref{prop map}. We start by going over a few facts about infinite dihedral subgroups of our Fuchsian group $\Gamma$. An infinite dihedral subgroup $D$ contains a unique index two infinite cyclic subgroup $T_D$. Since the normalizer in $\PSL_2\BR$ of a parabolic subgroup is torsion free, we get that the infinite cyclic subgroup of any infinite dihedral subgroup $D$ of $\Gamma$ is hyperbolic. It follows that $D$ acts on a geodesic $\CA_D\subset\BH^2$---the action $D\actson\CA_D$ is conjugated to the standard action of the infinite dihedral subgroup on the real line: the length of $D\bs\CA_D$ is the length $\ell(D)$ of the dihedral group and the geodesic $\gamma_D=T_D\bs\CA_D$ is the reciprocal geodesic associated to $D$. The stabilizer $\Stab_\Gamma(\CA_D)$ of the axis is also an infinite dihedral group---it is in fact the unique maximal dihedral subgroup of $\Gamma$ containing $D$.

Note now that if $D\notin\calD_\Gamma^{\max}$ then $\ell(\Stab_\Gamma(\CA_D))\le\frac 12\ell(D)$. Note also that an infinite dihedral group contains exactly $\left\lfloor\frac{3\cdot k}2\right\rfloor$ conjugacy classes of dihedral subgroups of index at most $k$. In plain language this means that every non-maximal infinite dihedral subgroup is the child of a dihedral subgroup of at most half the length, and that dihedral subgroups don't have may kids. Out of these two observations we get bounds for the number of maximal infinite dihedral subgroups of bounded length:

\begin{lem}\label{lem most maximal}
If $\epsilon_0<\ell(D)$ for every $D\in\calD_\Gamma$ then we have
$$\vert\Gamma\bs\calD_\Gamma(L)\vert\ge\vert\Gamma\bs\calD_\Gamma^{\max}(L)\vert\ge\vert\Gamma\bs\calD_\Gamma(L)\vert-\frac{3\cdot L}{2\cdot \epsilon_0}\cdot\vert\Gamma\bs\calD_\Gamma(2^{-1}\cdot L)\vert$$
for all $L>0$.\qed
\end{lem}

Continuing with generalities about infinite dihedral subgroups note that any such $D\subset\Gamma$ is generated by two distinct involutions $\sigma$ and $\bar\sigma$ fixing the axis $\CA_D$. In fact, there are precisely two $D$-conjugacy classes of ordered pairs of involutions generating $D$, namely $(\sigma,\bar\sigma)$ and $(\bar\sigma,\sigma)$. 

In the oposite direction suppose that $\sigma\neq\bar\sigma\in\Gamma$ are distinct involutions. Then the group $D=\langle\sigma,\bar\sigma\rangle$ they generate is infinite dihedral and $\CA_D$ is the infinite geodesic passing through the unique fixed points $p_\sigma$ and $p_{\bar\sigma}$ of $\sigma$ and $\bar\sigma$ respectively. In those terms, the length of the dihedral group is given by 
\begin{equation}\label{eq length}
\ell(\langle\sigma,\bar\sigma\rangle)=d_{\BH^2}(p_\sigma,p_{\bar\sigma}).
\end{equation}

All of this gives us a way to parametrize the set of all infinite dihedral subgroups of $\Gamma$. As all along let $\CI_\Gamma\subset\Gamma$ be the set of all involutions in $\Gamma$, that is of all elements of order two. From the discussion above we get surjectivity of the map
\begin{equation}\label{eq parametrisation}
\CI_\Gamma\times\CI_\Gamma\setminus\Delta\to\calD_\Gamma,\ \ (\sigma,\bar\sigma)\mapsto\langle\sigma,\bar\sigma\rangle
\end{equation}
where $\Delta$ is the diagonal in $\CI_\Gamma\times\CI_\Gamma$. The group $\Gamma$ acts on $\CI_\Gamma$ by conjugation. The map \eqref{eq parametrisation} is equivariant under this action and the induced map
\begin{equation}\label{eq parametrisation quotient}
\Gamma\bs(\CI_\Gamma\times\CI_\Gamma\setminus\Delta)\to\Gamma\bs\calD_\Gamma
\end{equation}
is surjective. Recall that, as we pointd out earlier, every ordered pair of involutions generating the infinite dihedral group $D=\langle\sigma,\bar\sigma\rangle$ is conjugated, within $D$, to either $(\sigma,\bar\sigma)$ or $(\bar\sigma,\sigma)$. It follows that the map \eqref{eq parametrisation quotient} is at worst 2-to-1 and that it is exactly 2-to-1 over the set of self-normalizing infinite dihedral subgroups. We record these facts for later use:

\begin{lem}
The map \eqref{eq parametrisation quotient} is at most 2-to-1. Moreover, conjugacy classes of maximal infinite dihedral subgroups have exactly two preimages.\qed
\end{lem}

Recall now that we are assuming that $\Gamma$ is finitely generated. This implies that it has only finitely many conjugacy classes of finite order elements and hence that the set $\Gamma\bs\CI_\Gamma$ is finite. Let $\CJ\subset\CI_\Gamma$ be a subset consisting of one representative of every $\Gamma$-conjugacy class. The map
\begin{equation}\label{eq quotient}
\bigsqcup_{\sigma\in\CJ}\big(\{\sigma\}\times(\CI_\Gamma\setminus\{\sigma\})\big)\to \Gamma\bs(\CI_\Gamma\times\CI_\Gamma\setminus\Delta)
\end{equation}
sending $(\sigma,\bar\sigma)$ to its conjugacy class is surjective and its restriction to the set $\{\sigma\}\times(\CI_\Gamma\setminus\{\sigma\})$ has fibers of cardinality equal to that of the normalizer $\CN_\Gamma(\sigma)$ of $\sigma$ in $\Gamma$.

Composing the maps \eqref{eq parametrisation quotient} and \eqref{eq quotient} we get thus a surjective map
\begin{equation}\label{eq map not complete}
\pi:\bigsqcup_{\sigma\in\CJ}\big(\{\sigma\}\times(\CI_\Gamma\setminus\{\sigma\})\big)\to\Gamma\bs\calD_\Gamma
\end{equation}
Let us recap what we can say about the cardinality of the fibers of \eqref{eq map not complete}. First, the preimage of $D\in\Gamma\bs\calD_\Gamma$ under \eqref{eq parametrisation quotient} has at most two points $(\sigma,\bar\sigma)$ and $(\bar\sigma,\sigma)$, with equality if $D$ is maximal. Now, the conjugacy class of $(\sigma,\bar\sigma)$ has $\vert\CN_\Gamma(\sigma)\vert$ preimages under \eqref{eq quotient}, and that of $(\bar\sigma,\sigma)$ has $\vert\CN_\Gamma(\bar\sigma)\vert$ preimages. Altogether we get that:

\begin{lem}\label{lem bound fiber}
For every conjugacy class of infinite dihedral subgroups $D=\langle\sigma,\bar\sigma\rangle\in\calD_\Gamma$ of $\Gamma$ we have
$$\vert\pi^{-1}(\langle\sigma,\bar\sigma\rangle)\vert\le \vert\CN_\Gamma(\sigma)\vert+\vert\CN_\Gamma(\bar\sigma)\vert$$
with equality if $D\in\calD_\Gamma^{\max}$ is maximal.\qed
\end{lem}

We are now ready to prove Proposition \ref{prop map}:

\begin{proof}[Proof of Proposition \ref{prop map}]
The basic observation needed to relate the statement of Proposition \ref{prop map} with what we have been discussing so far is that each involution $\sigma\in\CI_\Gamma$ is uniquely determined by its fixed points $p_\sigma\in\BH^2$. From this point of view, the map \eqref{eq map not complete} can be rewritten as 
\begin{equation}\label{eq map cloud1}
\begin{split}
\pi:&\bigsqcup_{(\sigma,\bar \sigma)\in\CJ\times\CJ}(\Gamma\cdot p_{\bar \sigma}\setminus\{p_\sigma\})\to\calD_\Gamma(L)\\
\pi:&(\sigma,\bar \sigma,\gamma\cdot p_{\bar \sigma})\mapsto\Gamma\text{-conjugacy class of }\langle \sigma,\gamma\bar \sigma\gamma^{-1}\rangle
\end{split}
\end{equation}
The map $\pi_L$ in \eqref{eq map cloud}, in the statement of the proposition, is just the restriction of this map to the set $\bigsqcup_{(\sigma,\bar \sigma)\in\CJ\times\CJ}\left(\Gamma\cdot p_{\bar \sigma}\cap B^*(p_\sigma,L)\right)$. Now we get from \eqref{eq length} that (groups in the conjugacy class of) the dihedral group $\pi(\sigma,\bar \sigma,\gamma p_{\bar \sigma})$ have length $d_{\BH^2}(p_\sigma,\gamma p_{\bar \sigma})$. It follows that the map $\pi_L$ in \eqref{eq map cloud} takes values in the desired set, and sujectivity follows from the surjectivity of $\pi$. The final claim of the proposition follows also automatically from Lemma \ref{lem bound fiber}.
\end{proof}

\section{}\label{sec sarnak}

In this section we prove Theorem \ref{sat sarnak} from the introduction. We restate it here for the convenience of the reader:

\begin{named}{Theorem \ref{sat sarnak}}
For every lattice $\Gamma\subset\PSL_2\BR$ with 2-torsion we have
$$\vert\Gamma\bs\calD_\Gamma(L)\vert\sim \frac{C(\Gamma)}{\vert\chi^{or}(\Gamma\bs\BH^2)\vert}\cdot e^L$$
where $\chi^{or}(\Gamma\bs\BH^2)$ is the Euler characteristic of the orbifold $\Gamma\bs\BH^2$, where 
$$C(\Gamma)=\frac 14\cdot\left(\sum_{\sigma\in\Gamma\bs\CI_\Gamma}\frac 1{\vert\CN_\Gamma(\sigma)\vert}\right)^2$$
and $\sim$ means that the ratio between both quatities tends to $1$ when $L\to\infty$.
\end{named}
\begin{proof}
The key fact we will need is Delsarte's classical result \cite{Delsarte} that 
\begin{equation}\label{eq delsarte}
\vert\Gamma\cdot y\cap B(x,R)\vert\sim \frac{\vol(B(x,R))}{\vert\Stab_\Gamma(y)\vert\cdot\vol(\Gamma\bs\BH^2)}
\end{equation}
when $R\to\infty$. Here $\vol(\Gamma\bs\BH^2)$ is the volume of the given orbifold. Via Gau\ss-Bonnet we can restate this in terms of the orbifold Euler-characteristic
$$\vert\Gamma\cdot y\cap B(x,R)\vert\sim \frac{e^R}{2\cdot\vert\Stab_\Gamma(y)\vert\cdot\vert\chi^{or}(\Gamma\bs\BH^2)\vert},$$
where we have used that $\vol(B(x,R))=2\pi(\cosh(R)-1)\sim\pi\cdot e^R$. 
Plugging this into Corollary \ref{prop enumerate} and noting that $\Stab_\Gamma(p_\sigma)=\CN_\Gamma(\sigma)$ for every involution $\sigma$ we get
\begin{align}
\vert\Gamma\bs\calD_\Gamma(L)\vert&\lesssim c\cdot \frac{e^L}{\vert\chi^{or}(\Gamma\bs\BH^2)\vert}\label{eq bound0}&\text{ for some }c>0,\\
\vert\Gamma\bs\calD_\Gamma(L)\vert&\gtrsim C(\Gamma)\cdot \frac{e^L}{\vert\chi^{or}(\Gamma\bs\BH^2)\vert}\label{eq bound1},&\text{ and}\\
\vert\Gamma\bs\calD_\Gamma^{\max}(L)\vert&\lesssim C(\Gamma)\cdot \frac{e^L}{\vert\chi^{or}(\Gamma\bs\BH^2)\vert}& \label{eq bound2}
\end{align}
where 
\begin{equation}\label{eq c}
C(\Gamma)=\frac 12\cdot \sum_{(\sigma,\bar\sigma)\in\CJ\times\CJ}\frac 1{\vert\CN_\Gamma(\bar\sigma)\vert\cdot(\vert\CN_\Gamma(\sigma)\vert+\vert\CN_\Gamma(\bar\sigma)\vert)}
\end{equation}
Here $\lesssim$ and $\gtrsim$ mean that the inequlities hold assymptotically when $L\to\infty$. Anyways, from \eqref{eq bound0} and \eqref{eq bound1} we get that $\vert\Gamma\bs\calD_\Gamma(L)\vert$ grows coarsely as $e^L$. It thus follows from Lemma \ref{lem most maximal} that
\begin{equation}\label{eq most are max}
\vert\Gamma\bs\calD_\Gamma(L)\vert\sim \vert\Gamma\bs\calD_\Gamma^{\max}(L)\vert
\end{equation}
From \eqref{eq bound1} and \eqref{eq bound2} we get a lower bound for the left side and and upper bound for the right side by the same quantity. We thus get
$$\vert\Gamma\bs\calD_\Gamma(L)\vert\sim C(\Gamma)\cdot \frac{e^L}{\vert\chi^{or}(\Gamma\bs\BH^2)\vert}$$
To conclude, elementary algebra yields that $C(\Gamma)$ as defined in \eqref{eq c} can be rewritten as in the statement of the theorem.
\end{proof}

\begin{bem}
Since it is going to be of some importance, we want to emphasize that Lemma \ref{lem most maximal} together with the exponential growth of the number of infinite dihedral subgroups implies that the proportion of non-maximal elements in $\Gamma\bs\calD_\Gamma(L)$ decreases exponentially when $L\to\infty$. 
\end{bem}

The fact that most infinite dihedral subgroups are maximal comes in handy now. Indeed, recall that we can associate to the conjugacy class of an infinite dihedral subgroup $D$ a reciprocal geodesic $\gamma_D$, that this map is surjective, and that it is in fact a bijection from the set of conjugacy classes of maximal infinite dihedral subgroups to the set of primitive reciprocal geodesics. Since most infinite dihedral subgroups are maximal we get that counting reciprocal geodesics is asymptotically equivalent to counting conjugacy classes of infinite dihedral subgroups. Corollary \ref{kor reciprocal} from the introduction follows then immediately from Theorem \ref{sat sarnak} together with the relation \eqref{eq trace length} between lengths of infinite dihedral subgroups and lengths of the associated reciprocal geodesics. 

\begin{named}{Corollary \ref{kor reciprocal}}
For every lattice $\Gamma\subset\PSL_2\BR$ with 2-torsion we have
$$\vert\{\gamma\text{ reciprocal geodesics in }\Gamma\bs\BH^2\text{ with }\tr(\gamma)\le X\}\vert\sim\frac{C(\Gamma)}{\vert\chi^{or}(\Gamma\bs\BH^2)\vert}\cdot X$$
as $X\to\infty$. Here notation is as in Theorem \ref{sat sarnak}.\qed
\end{named}

\section{}\label{sec bourgain kontorovich}
Let us now turn our attention to Theorem \ref{sat bourgain kontorovich}, which we also restate here:

\begin{named}{Theorem \ref{sat bourgain kontorovich}}
Let $\Gamma$ be a lattice with 2-torsion. For every $\delta>0$ there is a compact set $K_\delta\subset\Gamma\bs\BH^2$ such that
$$\vert\Gamma\bs\{D\in\calD_\Gamma(L)\text{ with }D\bs\CA_D\subset K_\delta\}\vert> e^{(1-\delta)\cdot L}$$
for all $L>0$ large enough.
\end{named}

We will reduce this theorem to the fact that the lattice $\Gamma\subset\PSL_2\BR$ has, for any $\delta<1$, a convex cocompact subgroup $\Gamma_0$ with critical exponent 
$$\delta(\Gamma_0)=\lim_{L\to\infty}\frac 1L\cdot\log\vert\{y\in\Gamma_0\cdot x_0\text{ with }d_{\BH^2}(x_0,y)\le L\}\vert>\delta$$
Indeed the following is true:

\begin{lem}\label{lem get large delta}
Every lattice $\Gamma\subset\PSL_2\BR$ has a sequence of finitely generated subgroups $\Gamma_k\subset\Gamma$ without parabolic elements and with 
$$\lim_{k\to\infty}\delta(\Gamma_k)=1$$ 
If $\Gamma$ has 2-torsion then $\Gamma_k$ can be chosen to also have 2-torsion for all $k$.
\end{lem}

\begin{bei}
How do these groups look like for the modular group $\PSL_2\BZ$? Well, in this case one can take $\Gamma_k$ to be the subgroup generated by the set $\{\eta^i\sigma\eta^{-i}\text{ with }i=-k,\dots,k\}$ where $\eta,\sigma\in\PSL_2\BZ$ correspond to the M\"obious transformations $\eta(z)=z+2$ and $\sigma(z)=-z^{-1}$.
\end{bei}

Lemma \ref{lem get large delta} will not surprise anybody and we suspect that in one way or the other it might well have appeared already in the literature. We prove it below using some amount of (very classical) technology but it can be done using elementary means and we encourage the reader to try to do it by themselves. Anyways, before going any further let us use this lemma to settle Theorem \ref{sat bourgain kontorovich}:

\begin{proof}[Proof of Theorem \ref{sat bourgain kontorovich}]
We get from Lemma \ref{lem get large delta} a finitely generated subgroup $\Gamma'\subset\Gamma$, with 2-torsion, without parabolic elements, and with $\delta(\Gamma')> 1-\delta$. Finite generation and lack of parabolics imply that $\Gamma'$ is convex cocompact and hence that there is a compact subset $K\subset\Gamma\bs\BH^2$ which contains $D\bs\CA_D$ for every infinite dihedral group whose conjugacy class admits a representative contained in $\Gamma'$.

Fix now an involution $\sigma\in\Gamma'\subset\Gamma$, choose the set $\CJ$ of representatives of $\Gamma\bs\CI_\Gamma$ in such a way that $\sigma\in\CJ$, and restrict the map $\pi_L$ in Proposition \ref{prop map} to the set $\{(\sigma,\sigma)\}\times(\Gamma'\cdot p_{\sigma})$. From the proposition we get that this map is at most $2\cdot\vert\CN_\Gamma(\sigma)\vert$-to-1. This means that at least $\frac 1{2\cdot\vert\CN_\Gamma(\sigma)\vert}\vert\Gamma'\cdot p_\sigma\cap B^*(p_\sigma,L)\vert$ elements in $\Gamma\bs\calD_\Gamma(L)$ have representatives contained in $\Gamma'$. From the very definition of the critical exponent and from the bound $\delta(\Gamma')>1-\delta$ we get that the cardinality of $\Gamma'\cdot p_\sigma\cap B^*(p_\sigma,L)$ grows faster than $e^{(1-\delta)\cdot L}$. Altogether we get that, for large $L$, there are at least $e^{(1-\delta)\cdot L}$ elements $D$ in $\Gamma\bs\calD_\Gamma(L)$ with $D\bs\CA_D$ contained in $K$. We are done.
\end{proof}

Now, let us prove Lemma \ref{lem get large delta}:

\begin{proof}[Proof of Lemma \ref{lem get large delta}]
The proof has two different steps. In a first algebraic/topological step we give a sequence of groups $\Gamma_k$. Then we use the relation between bottom of the spectrum and critical exponent to show that the groups $\Gamma_k$ have critical exponent tending to $1$. 

Anyways, let us start. Since this is the case we are interested in we are going to assume that $\Gamma$ has both 2-torsion and parabolic elements. This assumption implies that $\Gamma$ contains a subgroup $H$ isomorphic to $\BZ*\BZ*\BZ/2\BZ$ where the two first free factors correspond to maximal parabolic subgroups of $\Gamma$. Now, $H$ is the intersection of the finite index subgroups of $\Gamma$ containing it \cite{Scott}. It follows thus that $\Gamma$ has a finite index subgroup $\Gamma'$ such that the associated orbifold $\Gamma'\bs\BH^2$ has at least two cusps and a cone point of order $2$. We will find our subgroups inside $\Gamma'$.

Let now $c_0,\dots,c_k$ be the cusps of $\Gamma'\bs\BH^2$ and let $\eta=\eta_1\cup\dots\cup\eta_k\subset\Gamma'\bs\BH^2$ be a simple arc system contained in the regular part of our orbifold, with $\eta_i$ joining $c_0$ and $c_i$. We orient all those arcs in such a way that $c_0$ is always the origin and denote by 
$$\alpha:\Gamma'\to\BZ$$
the homomorphism given as follows: represent $\gamma\in\Gamma'=\pi_1(\Gamma'\bs\BH^2)$ by an oriented loop in the regular part of $\Gamma'\bs\BH^2$ and let $\alpha(\gamma)$ be the algebraic intersection number of that loop with the arc system $\eta=\eta_1\cup\dots\cup\eta_k$. It is a surjective homomorphism and by construction no element in 
$\Gamma''=\ker(\alpha)$ is parabolic and $\Gamma''$ has 2-torsion.

Let now $\Gamma_k\subset\Gamma''$ be any sequence of finitely generated subgroups with $\Gamma_k\subset\Gamma_{k+1}$ for all $k$ and with $\Gamma''=\cup_k\Gamma_k$. These are our groups and all that is left to argue is that $\delta(\Gamma_k)\to 1$ when $k$ grows.

\begin{claim}\label{claim bla}
$\lim_{k\to\infty}\delta(\Gamma_k)=1$.
\end{claim}

To establish this claim we will make use of a result of Patterson \cite{Patterson} and Sullivan \cite{Sullivan} asserting that for any discrete group $G\subset\PSL_2\BR$ with $\lambda_0(G\bs\BH^2)<\frac 14$ the critical exponent is given by the formula:
\begin{equation}\label{eq sullivan}
\delta(G)=\frac 12+\sqrt{\frac 14-\lambda_0(G\bs\BH^2)}.
\end{equation}
Note that $\lambda_0(\Gamma'\bs\BH^2)=0$ because $\Gamma'$ is a lattice. Now, since $\Gamma''\trianglelefteq\Gamma'$ is normal with $\Gamma'/\Gamma''\simeq\BZ$ amenable we get from \cite{Brooks}, or rather from \cite{Ballman}, that $\lambda_0(\Gamma''\bs\BH^2)=\lambda_0(\Gamma'\bs\BH^2)=0$. This means that for all $\epsilon>0$ there is a compactly supported function $f\in C^\infty_c(\Gamma''\bs\BH^2)$ with Rayleigh quotient $\CR(f)\le\epsilon$. Now, if $\Sigma$ is any compact connected subsurface containing the support of $f$ there is $k_0$ such that $\Gamma_k$ is contained in $\pi_1(\Sigma)$ for all $k\ge k_0$. This means that the surface $\Sigma$ lifts under the cover $\Gamma_k\bs\BH^2\to\Gamma''\bs\BH^2$. Lifting the function $f$ we get a function on $\Gamma_k\bs\BH^2$ which still has Rayleigh quotient less than $\epsilon$. In particular, 
$$\lambda_0(\Gamma_k\bs\BH^2)\le\epsilon$$
for all $k\ge k_0$. Claim \ref{claim bla} follows now from \eqref{eq sullivan}.
\end{proof}

\section{}\label{sec uniform}

In this section we prove Theorem \ref{sat equidistribution modular surface}, that is the equidistribution of reciprocal geodesics, making use of Proposition \ref{prop sweedy pie} below, a different equidistribution result which no expert will find surprising and which we prove in the next section. Consider namely for $x,y\in\BH^2$ the measures 
\begin{equation}\label{eq chichu is the best}
\tilde\mu_L^{x,y}=\sum_{z\in\Gamma\cdot x\cap B^*(y,L)}\overrightarrow{yz}
\end{equation}
where $\overrightarrow{yz}$ is the measure on $T^1\BH^2$ obtained by integrating, with respect to arc length, along the lift to the unit tangent bundle of the geodesic arc from $y$ to $z$. In the next section we will prove:

\begin{prop}\label{prop sweedy pie}
Let $\Gamma\subset\PSL_2\BR$ be a lattice and $\tilde\mu_L^{x,y}$ as in \eqref{eq chichu is the best}. For any compactly supported function $f\in C_c(T^1\Gamma\bs\BH^2)$ and any two $x,y\in\BH^2$ we have
$$\lim_{L\to\infty}\frac 1{\Vert\tilde\mu_L^{x,y}\Vert}\int \tilde f\ d\tilde\mu^{x,y}_L=\int f\ d\vol$$
where $\tilde f$ is the lift of $f$ to $T^1\BH^2$.
\end{prop}

Assuming Proposition \ref{prop sweedy pie} for the time being, we prove Theorem \ref{sat equidistribution modular surface}:

\begin{named}{Theorem \ref{sat equidistribution modular surface}}
If $\Gamma\subset\PSL_2\BR$ is a lattice which has 2-torsion then the measures $\mu_L$ as in \eqref{eq measures} converge projectively to the Liouville probability measure $\vol$. More precisely we have 
$$\lim_{L\to\infty}\frac 1{\Vert\mu_L\Vert}\int f\ d\mu_L=\int f\ d\vol$$
for every compactly supported continuous function on $T^1\Gamma\bs\BH^2$.
\end{named}

\begin{proof}
The measure $\mu_L$ is the measure on $T^1\Gamma\bs\BH^2$ obtained by integrating over the reciprocal geodesics in $\Gamma\bs\BH^2$ associated to infinite dihedral groups with length at most $L$. By Proposition \ref{prop map} we have a map
\begin{equation*}
\begin{split}
\pi_L:&\bigsqcup_{(\sigma,\bar \sigma)\in\CJ\times\CJ}\left(\Gamma\cdot p_{\bar \sigma}\cap B^*(p_\sigma,L)\right)\to\Gamma\bs\calD_\Gamma(L)\\
\pi_L:&(\sigma,\bar \sigma,\gamma\cdot p_{\bar \sigma})\mapsto\Gamma\text{-conjugacy class of }\langle \sigma,\gamma\bar \sigma\gamma^{-1}\rangle
\end{split}
\end{equation*}
of which we think as being a ``parametrization'' of the set of conjugacy classes of infinite dihedral groups. Recall that here $\CJ\subset\CI_\Gamma$ is a set of representatives of $\Gamma\bs\CI_\Gamma$, the set of conjugacy classes on involutions in $\Gamma$, and that $p_\sigma$ is the unique fixed point of the involution $\sigma$. With this notation consider  the measure
$$\hat\mu_L=\sum_{(\sigma,\bar \sigma)\in\CJ\times\CJ}\left[\frac 1{\vert \CN_\Gamma(\sigma)\vert+\vert\CN_\Gamma(\bar \sigma)\vert}
\sum_{q\in\Gamma\cdot p_{\bar \sigma}\cap B^*(p_\sigma,L)}\vec\gamma_{\pi_L(\sigma,\bar \sigma,q)}\right]$$
The measures $\mu_L$ and $\hat\mu_L$ are supported by exactly the same set of orbits of the geodesic flow by surjectivity of the map $\pi_L$. Moreover, on each component the two measures are multiples of each other, the multiple given by the cardinality of the fibers of $\pi_L$. In fact, the bound for the cardinality of the fibers in Proposition \ref{prop map} implies that these multiples are at most $1$, with equality whenever the dihedral group $\pi_L(\sigma,\bar \sigma,\gamma\cdot p_{\bar \sigma})$ is maximal and say has length at least $\frac 12L$. Since the proportion of non-maximal dihedral groups of length at most $L$ is exponentially small (compare with \eqref{eq most are max} and with the comment following the proof of Theorem \ref{sat sarnak}) we deduce that to prove that the measures $\frac 1{\Vert\mu_L\Vert}\mu_L$ converge to $\vol$ it suffices to show that
$$\lim_{L\to\infty}\frac 1{\Vert\hat\mu_L\Vert}\hat\mu_L=\vol.$$
Note now that the segment with end points $p_\sigma$ and $(\gamma\bar \sigma\gamma^{-1})(p_\sigma)$ projects to the geodesic $\gamma_{\pi_L(\sigma,\bar \sigma,\gamma\cdot p_{\bar \sigma})}$ under the covering map $\BH^2\to\Gamma\bs\BH^2$ and that the point $\gamma\cdot p_{\bar \sigma}$ is the midpoint of this segment. This means that the measure $\vec\gamma_{\pi_L(\sigma,\bar \sigma,\gamma\cdot p_{\bar \sigma})}$ is the projection to $T^1\Gamma\bs\BH^2$ of the measure $\overrightarrow{p_\sigma(\gamma p_{\bar \sigma})}+\overrightarrow{(\gamma p_{\bar \sigma})((\gamma\bar \sigma\gamma^{-1})(p_\sigma))}$ where, as earlier, $\overrightarrow{xy}$  is the measure on $T^1\BH^2$ obtained by integrating with respect to arc length along the lift to the unit tangent bundle of the geodesic arc from $x$ to $y$. Translating the second measure by $\gamma^{-1}$ we then get that the measure $\vec\gamma_{\pi_L(\sigma,\bar \sigma,\gamma\cdot p_{\bar \sigma})}$ is also the projection of the measure $\overrightarrow{p_\sigma(\gamma p_{\bar \sigma})}+\overrightarrow{p_{\bar \sigma}((\bar \sigma\gamma^{-1})(p_\sigma))}$. Altogether we get that $\hat\mu_L$ is the projection to $T^1\Gamma\bs\BH^2$ of the measure
$$\tilde\mu_L=2\cdot\sum_{(\sigma,\bar \sigma)\in\CJ\times\CJ}\left[\frac 1{\vert \CN_\Gamma(\sigma)\vert+\vert\CN_\Gamma(\bar \sigma)\vert}\sum_{q\in\Gamma\cdot p_{\bar \sigma}\cap B^*(p_\sigma,L)}\overrightarrow{p_\sigma q}\right]$$
on $T^1\BH^2$. With the same notation as in \eqref{eq chichu is the best} we can rewrite this as 
$$\tilde\mu_L=2\cdot\sum_{(\sigma,\bar \sigma)\in\CJ\times\CJ}\left[\frac 1{\vert \CN_\Gamma(\sigma)\vert+\vert\CN_\Gamma(\bar \sigma)\vert}\cdot 
\tilde\mu_L^{p_{\bar \sigma},p_\sigma}\right]$$
Proposition \ref{prop sweedy pie} asserts that the projections of the measures $\tilde\mu_L^{p_{\bar \sigma},p_\sigma}$ converge projectively to $\vol$ when $L\to\infty$. It follows that the same is true for $\hat\mu_L$, the projection of $\tilde\mu_L$. We are done.
\end{proof}

It remains to prove Proposition \ref{prop sweedy pie}.

\section{}\label{sec mixing}

In this section we prove Proposition \ref{prop sweedy pie}. Let us fix from now on $x,y\in\BH^2$ and $f\in C_c(T^1\Gamma\bs\BH^2)$, say with $\Vert f\Vert_\infty\le 1$, and note that it suffices to prove that for all $\delta>0$ there is $L_0$ with 
$$\left\vert\frac 1{\Vert\tilde\mu_L^{x,y}\Vert}\int \tilde f\ d\tilde\mu^{x,y}_L-\int f\ d\vol\right\vert<\delta$$
for all $L\ge L_0$. Given such a $\delta$ we choose $\epsilon>0$ with $\vert\tilde f(p)-\tilde f(q)\vert\le\delta$ for all $p,q\in\BH^2$ wich are at most at distance $10\epsilon$ of each other---this is possible because $\tilde f$ is the lift of the compactly supported function $f$. Note also that since $\epsilon>0$ can be reduced as much as we want we can think of $L=k\cdot\epsilon$ being an integer multiple of $\epsilon$. This means that it suffices to prove that
\begin{equation}\label{eq claim prop}
\limsup_{k\to\infty}\left\vert\frac 1{\Vert\tilde\mu_{k\cdot\epsilon}^{x,y}\Vert}\int \tilde f\ d\tilde\mu^{x,y}_{k\cdot\epsilon}-\int f\ d\vol\right\vert<\delta
\end{equation}
Now, slicing the ball of radius $k\cdot\epsilon$ as the union
$$B(y,k\cdot\epsilon)=\cup_{r=0}^{k-1}A(y,r\cdot\epsilon,\epsilon)$$
of concentric annuli $A(y,r\cdot\epsilon,\epsilon)=B(y,(r+1)\cdot\epsilon)\setminus B(y,r\cdot\epsilon)$, we write our measure as
$$\tilde\mu^{x,y}_{k\cdot\epsilon}=\sum_{r=0}^{k-1}\tilde\nu^{x,y}_{r}\ \ \ \ \text{ where }\ \ \ \ 
\tilde\nu^{x,y}_{r}=\sum_{z\in\Gamma\cdot x\cap A(y,r\cdot\epsilon,\epsilon)}\overrightarrow{yz}$$
Evidently, \eqref{eq claim prop} follows if we prove that
\begin{equation}\label{eq claim2}
\limsup_{k\to\infty}\left\vert\frac 1{\Vert\tilde\nu_{k}^{x,y}\Vert}\int \tilde f\ d\tilde\nu^{x,y}_{k}-\int f\ d\vol\right\vert<\delta
\end{equation}
We are not done yet decomposing our measures. Consider namely
$$\tilde\nu^{x,y}_{k}=\sum_{r=0}^{k}\tilde\lambda^{x,y}_{k,r}\ \ \ \ \text{ where }\ \ \ \ \tilde\lambda^{x,y}_{k,r}=\tilde\nu^{x,y}_k\vert_{A(y,r\cdot\epsilon,\epsilon)}$$
is for $r=0,\dots,k$ the restriction of $\tilde\nu_k^{x,y}$ to the annulus $A(y,r\cdot\epsilon,\epsilon)$. 
For fixed $k$, all of the measures $\tilde\lambda^{x,y}_{k,r}$ (save possibly the one with $r=k$, where it might be smaller), have the same total measure
\begin{equation}\label{eq tired}
\Vert\tilde\lambda^{x,y}_{k,r}\Vert=\epsilon\cdot\vert\Gamma\cdot x\cap A(y,k\cdot\epsilon,\epsilon)\vert\sim \frac{\epsilon\cdot\vol(A(y,k\cdot\epsilon,\epsilon))}{\vert\Stab_{\Gamma}(x)\vert\cdot\vol(\Sigma)}
\end{equation}
where the claim about asymptotics follows for example from Delsarte's orbit counting result \eqref{eq delsarte}. Anyways, the point is that to prove that $\tilde\nu_k^{x,y}$ satisifes \eqref{eq claim2} it suffices to prove that the measures $\tilde\lambda_{k,r}^{x,y}$ satisfy the analogue claim for most $r$. More concretely, \eqref{eq claim2}, and hence \eqref{eq claim prop} and thus Proposition \ref{prop sweedy pie}, follows once we establish the following:

\begin{claim}\label{claim2}
There are $k_1$ and $k_2$ such that for all $k>k_1+k_2$ and all choices of $r_k\in[k_1,k-k_2]$ we have
$$\limsup_{k\to\infty}\left\vert\frac 1{\Vert\tilde\lambda_{k,r_k}^{x,y}\Vert}\int \tilde f\ d\tilde\lambda^{x,y}_{k,r_k}-\int f\ d\vol\right\vert<\delta$$
\end{claim}

The role of $k_1$ is to ensure that the spheres of radius $r_k\cdot\epsilon$ around $x$ are well-mixed. Recall indeed that we can identify $T^1\BH^2$ with $\PSL_2\BR$ and that when doing so the geodesic flow becomes right multiplication by diagonal matrices $g_t\in\SL_2\BR$ with entries $e^{\pm\frac 12t}$. Mixing of the geodesic flow of $\Gamma\bs\BH^2$ (see \cite[III.2.3]{Bekka-Meyer}) implies that the projection of (the outer normal of) the spheres $S_t(y)=(T^1_y\BH^2)\cdot g_t$ gets equidistributed in $T^1\Gamma\bs\BH^2$ (see \cite[III.3.3]{Bekka-Meyer}). It follows that there is some $k_1$ with
\begin{equation}\label{eq spheres equidistribute}
\left\vert\int_{S_{r\cdot\epsilon}(y)}\tilde f - \int f\ d\vol\right\vert<\delta\ \ \ \ \text{ for all }r>k_1.
\end{equation}
Suppose from now on that we fix some $r>k_1$ and cut the sphere $S_{r\cdot\epsilon}(y)$ into segments $I_1,I_2,\dots,I_N$ of length $\ell(I_i)\in[\epsilon,2\epsilon]$ for all $i$. Denote then by
$$U_i=\cup_{t\in[0,\epsilon]} I_i\cdot g_t$$
the little surface area obtained by pushing $I_i$ via the geodesic flow for time $\epsilon$. By the choice of $\epsilon$ we have that
$$\sup_i\sup_{p,q\in U_i}\vert f(p)-f(q)\vert<\delta$$
This means that, choosing for all $i$ some point $p_i\in I_i$, we have
$$\left\vert\int_{S_{r\cdot\epsilon}(y)}\tilde f-\frac 1{\ell(S_{r\cdot\epsilon}(y))}\sum_i f(p_i)\cdot\ell(I_i)\right\vert<\delta$$
and similarly
$$\left\vert\frac 1{\Vert\tilde\lambda_{k,r}^{x,y}\Vert}\int \tilde f\ d\tilde\lambda_{k,r}^{x,y}-\frac 1{\Vert\tilde\lambda_{k,r}^{x,y}\Vert}\sum_i f(p_i)\cdot\tilde\lambda_{k,r}^{x,y}(U_i)\right\vert<\delta$$
These two bounds, together with \eqref{eq spheres equidistribute} and the assumption that $\Vert f\Vert_\infty\le 1$ imply that
\begin{align*}
\left\vert\frac 1{\Vert\tilde\lambda_{k,r}^{x,y}\Vert}\int\tilde f\ d\tilde\lambda_{k,r}^{x,y}\right.&-\left.\int f\ d\vol\right\vert\le\\
&\le\delta+\left\vert\frac 1{\Vert\tilde\lambda_{k,r}^{x,y}\Vert}\int\tilde f\ d\tilde\lambda_{k,r}^{x,y}-\int_{S_{r\cdot\epsilon}(y)}\tilde f\,\right\vert\\
&< 3\delta + \sum_{\tiny\begin{array}{c}i\in\{1,\ldots, N\}\text{ with}\\ U_i\cap\supp(f)\neq\emptyset\end{array}}\left\vert\frac{\ell(I_i)}{\ell(S_{r\cdot\epsilon}(y))}-\frac{\tilde\lambda_{k,r}^{x,y}(U_i)}{\Vert\tilde\lambda_{k,r}^{x,y}\Vert}
\right\vert
\end{align*}
Note that the proportion of $\lambda_{k,r}^{x,y}$ in $U_i$ is given by
$$\frac{\tilde\lambda_{k,r}^{x,y}(U_i)}{\Vert\tilde\lambda_{k,r}^{x,y}\Vert}=\frac{\vert\Gamma\cdot x\cap [U_i\cdot g_{(k-r)\epsilon}]\vert}{\vert\Gamma\cdot x\cap A(y,k\cdot\epsilon,\epsilon)\vert}$$
where $[V]\subset\BH^2$ denotes the image of $V\subset T^1\BH^2$ under the standard projection. 

Now, the exact same argument used to prove the equidistribution of spheres \cite[III.3.3]{Bekka-Meyer} shows that the spherical segments $I_i\cdot g_{(k-r)\epsilon}$ also get equidistributed when $(k-r)\to\infty$. Then, the argument allowing to recover Delsarte's assymptotics \eqref{eq delsarte} from the equidistribution of spheres (see \cite[III.3.5]{Bekka-Meyer}) yields that 
\begin{equation}\label{eq delsarte cheese}
\vert\Gamma\cdot x\cap [U_i\cdot g_{(k-r)\epsilon}]\vert\sim\frac{\vol([U_i\cdot g_{(k-r)\epsilon}])}{\vert\Stab_{\Gamma}(x)\vert\cdot\vol(\Sigma)}
\end{equation}
when $(k-r)\to\infty$. It follows thus from \eqref{eq tired} and \eqref{eq delsarte cheese} that
$$\frac{\tilde\lambda_{k,r}^{x,y}(U_i)}{\Vert\tilde\lambda_{k,r}^{x,y}\Vert}\sim\frac{\vol([U_i\cdot g_{(k-r)\epsilon}])}{\vol(A(y,k\cdot\epsilon,\epsilon))}$$
Now, since $f$ has compact support and since the lengths of the segments $I_i$ are pinched between two positive constants, we get that the last assymptotic statement holds uniformly for all $i$ with $I_i\cap\supp(\tilde f)$, meaning that there is $k_2$ with
$$\frac{(1- \delta)\cdot \vol([U_i\cdot g_{(k-r)\epsilon}])}{\vol(A(y,k\cdot\epsilon,\epsilon))}\le\frac{\tilde\lambda_{k,r}^{x,y}(U_i)}{\Vert \tilde\lambda_{k,r}^{x,y}\Vert}\le \frac{(1+\delta)\cdot \vol([U_i\cdot g_{(k-r)\epsilon}])}{\vol(A(y,k\cdot\epsilon,\epsilon))}$$
for all $i$ such that $U_i\cap\supp(f)\neq\emptyset$ and all $k-r\geq k_2$. Given that $\epsilon$ is fixed (and very small) we get that the ratio between volumes is comparable to the ratio between lengths. Also, the length of $I_i$ and the length of $S_{r\cdot\epsilon}(y)$ grow at exactly the same rate when we apply the geodesic flow. Combining these two facts we get 
$$\frac{(1- \delta-2\epsilon)\cdot \ell(I_i)}{\ell(S_{r\cdot\epsilon}(y))}\le\frac{\tilde\lambda_{k,r}^{x,y}(U_i)}{\Vert \tilde\lambda_{k,r}^{x,y}\Vert}\le \frac{(1+\delta+2\epsilon)\cdot \ell(I_i)}{\ell(S_{r\cdot\epsilon}(y))}$$
for all $i$ such that $U_i\cap\supp(f)\neq\emptyset$ and all $k\ge k_2$. This implies then that 
$$\sum_{\tiny\begin{array}{c}i\in\{1,\ldots,N\}\text{ with}\\ U_i\cap\supp(f)\neq\emptyset\end{array}}\left\vert\frac{\ell(I_i)}{\ell(S_{r\cdot\epsilon}(y))}-\frac{\tilde\lambda_{k,r}^{x,y}(U_i)}{\Vert\tilde\lambda_{k,r}^{x,y}\Vert}
\right\vert\le \delta+2\epsilon$$
and hence that 
$$\left\vert\int f\ d\vol-\frac 1{\Vert\lambda_{k,r}^{x,y}\Vert}\int \tilde f\ d\tilde\lambda_{k,r}^{x,y}\right\vert< 4\delta+2\epsilon$$
meaning that, up to changing one $\delta$ for another and after choosing $\epsilon$ really really small, we have proved Claim \ref{claim2}. Having proved the claim we have also proved Proposition \ref{prop sweedy pie}.\qed

\end{document}